\documentclass{amsart}

\usepackage{amsfonts}
\usepackage{amssymb}

\numberwithin{equation}{section}

\theoremstyle{plain}

\newtheorem{theorem}[subsection]{Theorem}
\newtheorem{corollary}[subsection]{Corollary}
\newtheorem{proposition}[subsection]{Proposition}
\newtheorem{lemma}[subsection]{Lemma}

\theoremstyle{definition}

\newtheorem*{question}{Question}

\renewcommand{\leq}{\leqslant}
\renewcommand{\geq}{\geqslant}

\renewcommand{\Re}{\mathop{\rm Re}\nolimits}
\renewcommand{\Im}{\mathop{\rm Im}\nolimits}

\newcommand{\wh}{\widehat}

\newcommand{\Z}{\mathbb{Z}}

\newcommand{\C}{\mathbb{C}}

\newcommand{\T}{\mathbb{T}}

\begin{document}

\title{Difference sets and the primes}

\author{Imre Z. Ruzsa}
\address{Alfr\'ed R\'enyi Institute of Mathematics\\
Hungarian Academy of Sciences\\
     Budapest, Pf. 127\\
     H-1364 Hungary
} \email{ruzsa@renyi.hu}

\author{Tom Sanders}
\address{Department of Pure Mathematics and Mathematical Statistics\\
University of Cambridge\\
Wilberforce Road\\
Cambridge CB3 0WA\\
England } \email{t.sanders@dpmms.cam.ac.uk}

\begin{abstract}
Suppose that $A \subset \{1,\dots,N\}$ is such that the difference
between any two elements of $A$ is never one less than a prime. We
show that $|A| = O(N\exp(-c\sqrt[4]{\log N}))$ for some absolute
$c>0$.
\end{abstract}

\maketitle

\section{Introduction}

In this paper we prove the following theorem.
\begin{theorem}\label{thm.mainresult}
Suppose that $N$ is an integer and $A \subset \{1,\dots,N\}$ is such
that the difference between any two elements of $A$ is never one
less than a prime. Then $|A| = O(N\exp(-c\sqrt[4]{\log N}))$ for
some absolute $c>0$.
\end{theorem}
The first explicit upper bounds for $|A|$ are due to S{\'a}rk{\"o}zy
\cite{AS} who showed, under the same hypotheses, that
\begin{equation*}
|A| = O(N\exp(-(2+o(1))\log \log \log N)).
\end{equation*}
Recently, in \cite{JL}, Lucier improved S{\'a}rk{\"o}zy's argument
using the formidable methods of Pintz, Steiger and Szemer{\'e}di
from \cite{JPWLSES}. Indeed, he showed that
\begin{equation*}
|A| = O(N\exp(-\omega(N)\log \log \log N)),
\end{equation*}
for some function $\omega(N)$ tending to infinity as $N$ tends to
infinity.\footnote{In fact he gets $\omega(N) \sim c\log \log \log
\log \log N$ for some absolute $c>0$.}

Complementing these results, the first author, in \cite{IZRP1},
showed that the bound on $|A|$ cannot be too small. Specifically,
that paper contains the following theorem.
\begin{theorem}
For any integer $N$, there is a set $A \subset \{1,\dots,N\}$ with
$|A| \geq \exp((\log 2 / 2 + o(1))\log N/ \log \log N)$ such that
the difference between any two elements of $A$ is never one less
than a prime.
\end{theorem}
The gap between the upper and lower bounds is, of course, incredibly
large, but even assuming the Generalized Riemann Hypothesis, which
would simplify our argument considerably, we could only get an upper
bound of the shape
\begin{equation*}
|A| = O(N\exp(-c\sqrt{\log N})),
\end{equation*}
for some absolute $c>0$. Thus we are lead to the following natural
question, asked by the first author in \cite{IZRUD}, and with which
we close our introduction.
\begin{question}
Assuming the Generalized Riemann Hypothesis, can one achieve a bound
of the shape $|A| = O(N^{1-c+o(1)})$, for some absolute $c>0$, in
Theorem \ref{thm.mainresult}?
\end{question}

\section{An outline of the paper}

The driving ingredient behind the proof of Theorem
\ref{thm.mainresult} is an energy increment argument which would be
made significantly easier if we had good estimates for the
distribution of primes in arithmetic progressions; the main work of
the paper comes from having to deal with the so called exceptional
zeros of $L$-functions. Our proof, then, begins in
\S\ref{sec.primenumberstheoremforaps}, by recalling some of the
tools necessary for dealing with such zeros.

The argument splits roughly into two cases. If there is no
exceptional zero then we have relatively good estimates for the
primes in progressions and the energy increment method has no
complications.

If there is an exceptional zero then, by averaging, we pass to a
progression of common difference equal to the modulus of the
character corresponding to the exceptional zero. We then conduct the
energy increment argument relative to this progression.

The two cases have separate major arcs estimates for the Fourier
transform of the primes; these are proved in
\S\ref{sec.majorarcestimate}. The minor arcs are then dealt with in
the usual, unified, manner in \S\ref{sec.minorarcestimate}.

It is possible to do away with the above bifurcation if one uses a
carefully weighted version of the primes. However, doing this adds
complications to the minor arcs estimates. Of course, once one has
put the work in to get these minor arc estimates the method can be
more easily transferred to other situations.

Having completed the basic Fourier estimates in
\S\S\ref{sec.primenumberstheoremforaps}, \ref{sec.majorarcestimate}
\verb!&! \ref{sec.minorarcestimate}, we prove some energy increment
results in \S\ref{sec.energyincrementlemmas} which are used in
\S\ref{sec.mainiterationlemma} to prove the main `iteration' lemma.
Finally, we complete the proof of Theorem \ref{thm.mainresult} in
\S\ref{sec.proofofmaintheorem}.

\section{Notation}\label{sec.notation}

Our main tool is the Fourier transform on $\Z$. We identify the dual
group of $\Z$ with $\T$ via the function $e(\theta):=\exp(2\pi i
\theta)$. Specifically, every additive character $\gamma:\Z
\rightarrow \C$ has the form $\gamma(x):=e(\theta x)$ for some
$\theta \in \T$. Now, we define the Fourier transform
$\wh{.}:\ell^1(\Z) \rightarrow L^\infty(\T)$ to be the map which
takes $f \in \ell^1(\Z)$ to
\begin{equation*}
\wh{f}(\theta):=\sum_{x \in \Z}{f(x)\overline{e(x \theta )}},
\end{equation*}
and similarly convolution to be the map $*:\ell^1(\Z) \times
\ell^1(\Z) \rightarrow \ell^1(\Z)$ which takes $f,g \in \ell^1(\Z)$
to
\begin{equation*}
f \ast g(x) = \sum_{x \in \Z}{f(x')g(x-x')}.
\end{equation*}

As usual with the Fourier transform on $\Z$ we shall decompose the
dual group into major and minor arcs. To this end suppose that
$\eta>0$, and $a$ and $q$ are positive integers. We
write\footnote{Technically elements of $\T$ are equivalence classes
and so $|.|$ is not well defined. We adopt the usual conventions in
this regard.}
\begin{equation*}
\mathfrak{M}_{a,q,\eta}:=\{\theta \in \T: |\theta - a/q| \leq
\eta\},
\end{equation*}
\begin{equation*}
\mathfrak{M}_{q,\eta}^*:=\bigcup\{\mathfrak{M}_{a,q,\eta}: 1 \leq a
\leq q \textrm{ and } (a,q)=1\},
\end{equation*}
and
\begin{equation*}
\mathfrak{M}_{q,\eta}:=\bigcup\{\mathfrak{M}_{a,q,\eta}: 1 \leq a
\leq q\}.
\end{equation*}
Often there will be a further parameter $Q$ with $q \leq Q$ in which
case we will usually have $\eta=1/qQ$ and write
\begin{equation*}
\mathfrak{M}_{a,q}:=\mathfrak{M}_{a,q,(qQ)^{-1}},
\mathfrak{M}_q^*:=\mathfrak{M}_{q,(qQ)^{-1}}^* \textrm{ and }
 \mathfrak{M}_q:=\mathfrak{M}_{q,(qQ)^{-1}}.
\end{equation*}
The quantity $Q$ will always be clear from the context.

\section{A prime number theorem for arithmetic
progressions}\label{sec.primenumberstheoremforaps}

In this section we shall develop the small amount of number theory
which we require. All the results we use are well known, although
they are not always stated in the most useful fashion. We shall
refer to the book \cite{HD} of Davenport.

Suppose that $x$ is real and $a$ and $q$ are positive integers. Then
we write
\begin{equation*}
\psi(x;q,a):=\sum_{\genfrac{}{}{0pt}{}{n \leq x}{n \equiv
{a}\pmod{q}}}{\Lambda(n)},
\end{equation*}
where $\Lambda$ is the usual von-Mangoldt function.

Estimating $\psi(x;q,a)$ is one of the central problems in analytic
number theory and to do so we introduce some auxiliary functions:
For a Dirichlet character $\chi$ define
\begin{equation*}
\psi(x,\chi):=\sum_{n \leq x}{\chi(n)\Lambda(n)}.
\end{equation*}
The analysis of $\psi(x,\chi)$ is, in turn, bound up in the analysis
of the zeros of the corresponding $L$-function, $L(s,\chi)$, which
is complicated by the possibility of a so called exceptional zero;
the following theorem limits the number of possible exceptions for a
given Dirichlet character.
\begin{theorem}\emph{(\cite[Chapter 14]{HD})}\label{thm.oneexceptionalzeropercharacter}
There is an absolute constant $c_1>0$ such that for any
non-principal Dirichlet character $\chi$ of modulus $q$, $L(s,\chi)$
has at most one zero in the region
\begin{equation*}
\Re s \geq 1- \frac{c_1}{\log q(|\Im s|+3)}.
\end{equation*}
This exceptional zero may only occur if $\chi$ is real, and then it
is a simple real zero.
\end{theorem}
As usual, the analysis of the zeros of $L(s,\chi)$ can be reduced to
the case when $\chi$ is primitive. Indeed, if $\chi$ has modulus $q$
and is induced by $\chi'$ then, by the Euler product formula, we
have
\begin{equation*}
L(s,\chi)=\prod_{p | q}{(1-\chi'(p)p^{-s})}L(s,\chi') \textrm{ for }
\Re s >1.
\end{equation*}
Analytic continuation then tells us that in the region $\Re s>0$ we
have $L(s,\chi)=0$ iff $L(s,\chi')=0$. Now, Landau showed that an
exceptional zero can only occur for at most one \emph{primitive}
Dirichlet character:
\begin{theorem}\emph{(\cite[Chapter 14]{HD})}\label{thm.onecharacterwithexceptionalzeros}
There is an absolute constant $c_2>0$ such that for any distinct
primitive real Dirichlet characters $\chi_1$ and $\chi_2$ with
moduli $q_1$ and $q_2$, and real zeros $\beta_1$ and $\beta_2$
respectively we have
\begin{equation*}
\min\{\beta_1,\beta_2\} \leq 1- \frac{c_2}{\log q_1q_2}.
\end{equation*}
\end{theorem}
Write $c_E = \min\{c_1,c_2\}$ and suppose that $D \geq 2$ and
$\chi$, a Dirichlet character, are given. We say that $\beta_\chi$
is an \emph{exceptional zero for $\chi$ at level $D$} if
\begin{equation*}
L(\beta_\chi,\chi)=0 \textrm{ and } \Re \beta_\chi \geq
1-\frac{c_E}{\log 3D}.
\end{equation*}
The following corollary is an immediate consequence of Theorems
\ref{thm.oneexceptionalzeropercharacter} and
\ref{thm.onecharacterwithexceptionalzeros}.
\begin{corollary}
Suppose that $D \geq 2$. Then there is at most one primitive
Dirichlet character $\chi_D$ and zero $\beta_D$ such that $\beta_D$
is an exceptional zero for $\chi_D$ at level $D$ and $\chi_D$ has
modulus $q_D \leq D$.
\end{corollary}
If it exists we call the Dirichlet character $\chi_D$ of the
corollary the \emph{exceptional Dirichlet character at level $D$}
and $\beta_D$ the \emph{exceptional zero at level $D$}. In this
event we shall need a bound on $(1-\beta_D)^{-1}$.
\begin{proposition}\emph{(\cite[Chapter 14]{HD})}\label{prop.effectiveboundonexceptionalzero}
Suppose that $D \geq 2$ and the exceptional Dirichlet character at
level $D$ exists and has zero $\beta_D$ and modulus $q_D$. Then
$(1-\beta_D)^{-1} = O(q_D^{1/2}\log^2q_D)$.
\end{proposition}
We require the following two prime number theorems.
\begin{theorem}\emph{(\cite[Chapter 20]{HD})}\label{thm.primenumbertheoreminaps}
There is an absolute constant $c_3>0$ such that if $D \geq 2$ and
$\chi$ is a non-principal Dirichlet character of modulus $q \leq D$,
then
\begin{enumerate}
\item \label{case.primenumbertheoreminapscase1}
if the exceptional Dirichlet character $\chi_D$ exists and $\chi$ is
induced by $\chi_D$ then for any real $x \geq 1$ we have
\begin{equation*}
\psi(x,\chi) = - \frac{x^{\beta_D}}{\beta_D} + O\left(x
\exp\left(-\frac{c_3 \log x}{\sqrt{\log x} + \log D}\right)(\log
D)^2\right)
\end{equation*}
where $\beta_D$ is the exceptional zero;
\item \label{case.primenumbertheoreminapscase2}
if the exceptional Dirichlet character $\chi_D$ does not exists or
$\chi$ is not induced by $\chi_D$ then for any real $x \geq 1$ we
have
\begin{equation*}
\psi(x,\chi) = O\left(x \exp\left(-\frac{c_3 \log x}{\sqrt{\log x} +
\log D}\right)(\log D)^2\right).
\end{equation*}
\end{enumerate}
\end{theorem}
\begin{theorem}\emph{(\cite[Chapter 20]{HD})}\label{thm.primenumbertheorem}
There is an absolute constant $c_4>0$ such that if $\chi'$ is the
principal Dirichlet character of modulus $q$, then for all real $x
\geq 1$ we have
\begin{equation*}
\psi(x,\chi')=x + O\left(x\exp(-c_4\sqrt{\log x}) + \log q\log
x\right).
\end{equation*}
\end{theorem}
Getting a handle on $\psi(x;q,a)$ is now done via the identity
\begin{equation}\label{eqn.fourierinversionforpsi}
\psi(x;q,a) =
\frac{1}{\phi(q)}\sum_{\chi}{\overline{\chi}(a)\psi(x,\chi)},
\end{equation}
where the summation is over all Dirichlet characters of modulus $q$.
We can now prove the following proposition which is to be regarded
as definitive for the terms \emph{$(D_1,D_0)$ is exceptional} and
\emph{$(D_1,D_0)$ is unexceptional}.
\begin{proposition}\label{prop.usefulprimenumbertheorem}
There is an absolute constant $c_5>0$ such that if $D_1\geq D_0 \geq
2$, then at least one of the following two possibilities holds.
\begin{enumerate}
\item \emph{($(D_1,D_0)$ is exceptional)} There is a character $\chi_{D}$ of modulus $q_{D} \leq D_0$
and a real $\beta_D$ with $(1-\beta_D)^{-1} =O(q_D^{1/2}\log^2q_D)$
such that for any real $x \geq 1$ and integers $a$ and $q$ with $1
\leq qq_D \leq D_1$ we have
\begin{eqnarray*}
\psi(x;qq_D,a) & = & \frac{\overline{\chi'}(a)x}{\phi(qq_D)} -
\frac{\overline{\chi'\chi_D}(a)x^{\beta_D}}{\phi(qq_D)\beta_D}\\ & &
+ O\left(x \exp\left(-\frac{c_5\log x}{\sqrt{\log x} + \log
D_1}\right)(\log D_1)^2\right),
\end{eqnarray*}
where $\chi'$ is the principal character of modulus $qq_D$.
\item \emph{($(D_1,D_0)$ is unexceptional)} For any real $x \geq 1$ and integers $a$ and $q$ with $1 \leq q \leq D_0$
we have
\begin{equation*}
\psi(x;q,a)=\frac{\overline{\chi'}(a)x}{\phi(q)}+ O\left(x
\exp\left(-\frac{c_5\log x}{\sqrt{\log x} + \log D_1}\right)(\log
D_1)^2\right),
\end{equation*}
where $\chi'$ is the principal Dirichlet character of modulus $q$.
\end{enumerate}
\end{proposition}
\begin{proof}
Let $c_5:=\min\{c_3,c_4\}$. We split into two cases according to
whether on not there is an exceptional character $\chi_D$ with
modulus at most $D_0$.

First, suppose that $\chi_D$ does exist, has zero $\beta_D$ and has
modulus $q_D \leq D_0$. By Proposition
\ref{prop.effectiveboundonexceptionalzero}, $(1-\beta_D)^{-1}$
satisfies the appropriate bound. Now suppose that $x \geq 1$ is real
and $a$ and $q$ are integers with $1 \leq qq_D \leq D_1$. Write
$\chi'$ for the principal character of modulus $qq_D$. There is
exactly one character of modulus $qq_D$ induced by $\chi_D$ and that
is $\chi'\chi_D$. For this character, by Theorem
\ref{thm.primenumbertheoreminaps} case
(\ref{case.primenumbertheoreminapscase1}), we have
\begin{equation*}
\psi(x,\chi'\chi_D)=-\frac{x^{\beta_D}}{\beta_D} + O\left(x
\exp\left(-\frac{c_5 \log x}{\sqrt{\log x} + \log D_1}\right)(\log
D_1)^2\right).
\end{equation*}
For all other non-principal characters $\chi$ we have, by Theorem
\ref{thm.primenumbertheoreminaps} case
(\ref{case.primenumbertheoreminapscase2}),
\begin{equation*}
\psi(x,\chi)= O\left(x \exp\left(-\frac{c_5 \log x}{\sqrt{\log x} +
\log D_1}\right)(\log D_1)^2\right).
\end{equation*}
Finally by Theorem \ref{thm.primenumbertheorem} we have
\begin{equation*}
\psi(x,\chi') = x +O\left(x\exp(-c_5\sqrt{\log x}) + \log D_1\log
x\right).
\end{equation*}
Inserting these into (\ref{eqn.fourierinversionforpsi}) gives the
first case of the proposition.

In the second case we suppose that either $\chi_D$ doesn't exist or,
if it does, then it has modulus greater than $D_0$. Now suppose that
$x \geq 1$ is real and $a$ and $q$ are integers with $1 \leq q \leq
D_0$. Since $q$ is smaller than the modulus of $\chi_D$, if it
exists at all, no character of modulus $q$ is induced by $\chi_D$,
and we can apply Theorem \ref{thm.primenumbertheoreminaps} case
(\ref{case.primenumbertheoreminapscase2}) to conclude that
\begin{equation*}
\psi(x,\chi) =  O\left(x \exp\left(-\frac{c_5 \log x}{\sqrt{\log x}
+ \log D_1}\right)(\log D_1)^2\right)
\end{equation*}
for every non-principal $\chi$ of modulus $q$. Once again Theorem
\ref{thm.primenumbertheorem} gives
\begin{equation*}
\psi(x,\chi') = x +O\left(x\exp(-c_5\sqrt{\log x}) + \log D_1\log
x\right),
\end{equation*}
for $\chi'$ the principal character of modulus $q$. Now inserting
these estimates into (\ref{eqn.fourierinversionforpsi}) we find
ourselves in the second case of the proposition.
\end{proof}

\section{The major arcs}\label{sec.majorarcestimate}

We are interested in the Fourier transform of the von-Mangoldt
function $\Lambda$ and some closely related functions. Suppose that
$N$ and $d$ are positive integers. We write
\begin{equation*}
\Lambda_{N,d}:=\begin{cases}  \Lambda(dx+1) & \textrm{ if } 1 \leq x
\leq N\\ 0 & \textrm{ otherwise}.
\end{cases}
\end{equation*}
We write $\Lambda_N$ as shorthand for $\Lambda_{N,1}$. There will be
two types of estimate for $\wh{\Lambda_{N,d}}$ depending on whether
or not a given pair of parameters $D_1\geq D_0 \geq 2$ is
exceptional or unexceptional. The reader may care to recall the
definition from Proposition \ref{prop.usefulprimenumbertheorem}.

Before we begin it will be useful to recall some standard
definitions; the reader unfamiliar with this material may wish to
consult the book \cite{HLMRCV}. For an integer $a$ and positive integer
$q$ the Ramanujan sum $c_q(a)$ is defined by
\begin{equation*}
c_q(a):=
\sum_{\genfrac{}{}{0pt}{}{h=1}{(h,q)=1}}^q{e\left(\frac{ha}{q}\right)},
\end{equation*}
and, moreover, $c_q(1)=\mu(q)$.

If positive integers $q$ and $d$ are coprime, write $m_{d,q}$ for a
solution to $m_{d,q}d + 1\equiv 0 \pmod q$. Then for any integers
$a,q$ and $d$ with $q$ and $d$ positive we put
\begin{equation*}
\tau_{a,d,q}:=\sum_{\genfrac{}{}{0pt}{}{m=0}{(md+1,q)=1}}^{q-1}{e\left(m\frac{a}{q}\right)}
= \begin{cases} c_q(a)e\left(-m_{d,q}\frac{a}{q}\right) & \textrm{
if } (d,q)=1\\ 0 & \textrm{ otherwise.}
\end{cases}
\end{equation*}
The proof of the equivalence of the sum with the expression in terms
of the Ramanujan sum is a simple change of variables.

The remainder of this section provides major arcs estimates for the
two cases when the pair $(D_1,D_0)$ is exceptional and
unexceptional.

\subsection{Exceptional pairs}

Throughout this subsection we assume that the pair $(D_1,D_0)$ is
exceptional. We begin by estimating $\wh{\Lambda_{N,d}}$ at a
rational with small denominator and then extend the range.
\begin{lemma}\label{lem.majorarcsatrationals}
There is an absolute constant $c_6>0$ such that for every set of
non-negative integers $N,a,q,d$ with $d_D | d$, $1 \leq dq \leq D_1$
and $N \geq 1$ we have
\begin{eqnarray*}
\wh{\Lambda_{N,d}}(a/q)  =  \frac{dN\tau_{a,d,q}}{\phi(d)\phi(q)} &
- &  \frac{(dN)^{\beta_D}\tau_{a,d,q}}{\phi(d)\phi(q)\beta_D}\\ & +
& O\left(ND_1^2 \exp\left(-\frac{c_6\log N}{\sqrt{\log N} + \log
D_1}\right)\right).
\end{eqnarray*}
\end{lemma}
\begin{proof}
Note the formula
\begin{eqnarray}\label{eqn.formulaforlambdahat}
\wh{\Lambda_{N,d}}(a/q) & = & \sum_{\genfrac{}{}{0pt}{}{x \leq
dN+1}{x \equiv 1 \pmod
d}}{\Lambda(x)e\left(\frac{a(x-1)}{dq}\right)}\\ \nonumber & =
& \sum_{m=0}^{q-1}{e\left(m \frac{a}{q}\right)\sum_{\genfrac{}{}{0pt}{}{x \leq dN+1}{\genfrac{}{}{0pt}{}{x \equiv 1 \pmod d}{\frac{x-1}{d} \equiv m \pmod q}}}{\Lambda(x)}}\\
\nonumber & = & \sum_{m=0}^{q-1}{e\left(m
\frac{a}{q}\right)\psi(dN+1;dq,md+1)}.
\end{eqnarray}
Since $(D_1,D_0)$ is exceptional we get a character $\chi_D$ of
modulus $d_D \leq D_0$ and a real $\beta_D$ with $(1-\beta_D)^{-1}
=O(d_D^{1/2}\log^2d_D)$ such that for any real $x \geq 1$ and
integers $a'$ and $q'$ with $1 \leq q'd_D \leq D_1$ we have
\begin{eqnarray}
\label{eqn.estimateforpsi}\psi(x;q'd_D,a') & = &
\frac{\overline{\chi'}(a')x}{\phi(q'd_D)} -
\frac{\overline{\chi'\chi_D}(a')x^{\beta_D}}{\phi(q'd_D)\beta_D}\\
\nonumber & & + O\left(x \exp\left(-\frac{c_5\log x}{\sqrt{\log x} +
\log D_1}\right)(\log D_1)^2\right),
\end{eqnarray}
where $\chi'$ is the principal character of modulus $q'd_D$. Now
suppose that $d_D | d$ and $1 \leq dq \leq D_1$. There are three
terms to consider when substituting (\ref{eqn.estimateforpsi}), with
$q'=dq/d_D$, $x=dN+1$ and $a'=md+1$, into
(\ref{eqn.formulaforlambdahat}). First,
\begin{eqnarray*}
\sum_{m=0}^{q-1}{e\left(m
\frac{a}{q}\right)\frac{\overline{\chi'}(md+1)(dN+1)}{\phi(dq)}} & =
&
\frac{(dN+1)}{\phi(dq)}\sum_{\genfrac{}{}{0pt}{}{m=0}{(md+1,dq)=1}}^{q-1}{e\left(m
\frac{a}{q}\right)}\\ & = &
\frac{(dN+1)\tau_{a,d,q}}{\phi(d)\phi(q)},
\end{eqnarray*}
recalling the definition of $\tau_{a,d,q}$ and the fact that it is
zero unless $(d,q)=1$. Secondly, we have the sum
\begin{equation*}
\sum_{m=0}^{q-1}{e\left(m
\frac{a}{q}\right)\frac{\overline{\chi'\chi_D}(md+1)(dN+1)^{\beta_D}}{\phi(dq)\beta_D}}.
\end{equation*}
Since $\chi_D$ has modulus $d_D$ which divides $d$ we conclude that
$\chi_D(md+1) = \chi_D(1) =1$ whatever the value of $m$, thus the
above sum is equal to
\begin{equation*}
\sum_{m=0}^{q-1}{e\left(m
\frac{a}{q}\right)\frac{\overline{\chi'}(md+1)(dN+1)^{\beta_D}}{\phi(dq)\beta_D}}
= \frac{(dN+1)^{\beta_D}\tau_{a,d,q}}{\phi(d)\phi(q)\beta_D}
\end{equation*}
by the same calculation as for the previous sum. Finally, we have an
error term
\begin{equation*}
O\left(q(dN+1)\exp\left(-\frac{c_5\log (dN+1)}{\sqrt{\log (dN+1)} +
\log D_1}\right)(\log D_1)^2\right),
\end{equation*}
which is certainly
\begin{equation*}
O\left(N D_1^2 \exp\left(-\frac{c_6\log N}{\sqrt{\log N} + \log
D_1}\right)\right)
\end{equation*}
for $c_6:=c_5/4$. Combining these terms yields the lemma.
\end{proof}

\begin{proposition}[Major arcs estimate for exceptional pairs]\label{prop.majorarcsestimateforexceptionalpairs}
For all non-negative integers $N,a,q,d$ with $d_D | d$, $1 \leq dq
\leq D_1$, $(a,q)=1$ and $N \geq 1$, and elements $\theta \in \T$ we
have
\begin{equation*}
|\wh{\Lambda_{N,d}}(\theta)| \leq
\frac{|\wh{\Lambda_{N,d}}(0)|}{\phi(q)}+ O\left((1+|\kappa|N)ND_1^2
\exp\left(-\frac{c_6\log N}{\sqrt{\log N} + \log D_1}\right)\right),
\end{equation*}
where $\kappa:= \theta - a/q$ and
\begin{equation*}
|\wh{\Lambda_{N,d}}(0)| \gg \frac{N}{\phi(d)} + O\left(ND_1^2
\exp\left(-\frac{c_6\log N}{\sqrt{\log N} + \log D_1}\right)\right).
\end{equation*}
\end{proposition}
\begin{proof}
Begin by applying Lemma \ref{lem.majorarcsatrationals} to get that
for every set of non-negative integers $x,a,q,d$ with $d_D | d$, $1
\leq dq \leq D$ and $1 \leq x \leq N$ we have
\begin{eqnarray}\label{eqn.generalestimateforlambda}
\wh{\Lambda_{x,d}}(a/q) & = & \frac{dx\tau_{a,d,q}}{\phi(d)\phi(q)}
- \frac{(dx)^{\beta_D}\tau_{a,d,q}}{\phi(d)\phi(q)\beta_D}\\
\nonumber & & + O\left(ND_1^2 \exp\left(-\frac{c_6\log N}{\sqrt{\log
N} + \log D_1}\right)\right).
\end{eqnarray}
In particular we have
\begin{equation}\label{eqn.zeroestimateforlambda}
\wh{\Lambda_{x,d}}(0) = \frac{dx}{\phi(d)} -
\frac{(dx)^{\beta_D}}{\phi(d)\beta_D} + O\left(ND_1^2
\exp\left(-\frac{c_6\log N}{\sqrt{\log N} + \log D_1}\right)\right),
\end{equation}
and hence
\begin{equation}\label{eqn.usefulestimateforlambda}
\wh{\Lambda_{x,d}}(a/q) =
\frac{\tau_{a,d,q}}{\phi(q)}\wh{\Lambda_{x,d}}(0) + O\left(ND_1^2
\exp\left(-\frac{c_6\log N}{\sqrt{\log N} + \log D_1}\right)\right).
\end{equation}

Observe, by telescoping, that
\begin{equation*}
\wh{\Lambda_{N,d}}(\theta) = \sum_{n
=1}^{N}{(\wh{\Lambda_{n,d}}(a/q) -
\wh{\Lambda_{n-1,d}}(a/q))e(\kappa n)}.
\end{equation*}
Integration by parts then tells us that
\begin{equation}\label{eqn.lambdahatbyparts}
\wh{\Lambda_{N,d}}(\theta) = \left[\wh{\Lambda_{x,d}}(a/q)e(\kappa
x)\right]_0^N - 2\pi i
\kappa\int_0^N{\wh{\Lambda_{x,d}}(a/q)e(\kappa x)dx}.
\end{equation}
We use (\ref{eqn.usefulestimateforlambda}) to estimate the right
hand side of this. The first term is
\begin{equation*}
\frac{\tau_{a,d,q}e(\kappa N)}{\phi(q)}\wh{\Lambda_{N,d}}(0) +
O\left(ND_1^2 \exp\left(-\frac{c_6\log N}{\sqrt{\log N} + \log
D_1}\right)\right).
\end{equation*}
We consider the second term on the right of
(\ref{eqn.lambdahatbyparts}) in two parts. First, note that
\begin{equation*}
2\pi i \kappa \int_0^N{\left(\wh{\Lambda_{x,d}}(a/q)
-\frac{dx\tau_{a,d,q}}{\phi(d)\phi(q)}+\frac{(dx)^{\beta_D}\tau_{a,d,q}}{\beta_D\phi(d)\phi(q)}\right)
e(\kappa x)dx}
\end{equation*}
is
\begin{equation*}
O\left(|\kappa|N^2D_1^2 \exp\left(-\frac{c_6\log N}{\sqrt{\log N} +
\log D_1}\right)\right)
\end{equation*}
by (\ref{eqn.generalestimateforlambda}). Secondly, note that
\begin{equation*}
2\pi i \kappa
\int_0^N{\left(\frac{dx\tau_{a,d,q}}{\phi(d)\phi(q)}-\frac{(dx)^{\beta_D}\tau_{a,d,q}}{\beta_D\phi(d)\phi(q)}\right).
e(\kappa x)dx}
\end{equation*}
is equal to
\begin{equation*}
\frac{\tau_{a,d,q}}{\phi(d)\phi(q)}\left[\left(dx-\frac{(dx)^{\beta_D}}{\beta_D}\right).
e(\kappa x)\right]_0^N - \frac{\tau_{a,d,q}}{\phi(d)\phi(q)}
\int_0^N{\left(d-d^{\beta_D}x^{\beta_D-1}\right). e(\kappa x)dx},
\end{equation*}
by integration by parts. The first term here is equal to
\begin{equation*}
\frac{\tau_{a,d,q}e(\kappa N)}{\phi(q)}\wh{\Lambda_{N,d}}(0) +
O\left(ND_1^2 \exp\left(-\frac{c_6\log N}{\sqrt{\log N} + \log
D_1}\right)\right),
\end{equation*}
by (\ref{eqn.zeroestimateforlambda}). So, combining what we have so
far we get that
\begin{eqnarray}\label{eqn.almostestimated}
\wh{\Lambda_{N,d}}(\theta) & = &
\frac{\tau_{a,d,q}}{\phi(q)\phi(d)}\int_0^N{(d-d^{\beta_D}x^{\beta_D-1})e(\kappa
x)dx}\\ \nonumber & & + O\left((1+|\kappa|N)ND_1^2
\exp\left(-\frac{c_6\log N}{\sqrt{\log N} + \log D_1}\right)\right).
\end{eqnarray}
Now, note that $d(1- (dx)^{\beta_D-1}) \geq 0$ if $dx \geq 1$, so
\begin{eqnarray*}
\left|\int_0^N{\left(d-d^{\beta_D}x^{\beta_D-1}\right). e(\kappa
x)dx}\right| & \leq &
\left|\int_1^N{\left(d-d^{\beta_D}x^{\beta_D-1}\right). e(\kappa
x)dx}\right|+ O(1)\\ & \leq
&\int_1^N{\left|d-d^{\beta_D}x^{\beta_D-1}\right|dx}+ O(1)\\
& = &\int_1^N{d-d^{\beta_D}x^{\beta_D-1}dx}+ O(1)\\ & = & dN -
\frac{(dN)^{\beta_D}}{\beta_D} + O(1).
\end{eqnarray*}
Thus we conclude that the integral in (\ref{eqn.almostestimated}) is
bounded above in absolute value by
\begin{equation*}
|\wh{\Lambda_{N,d}}(a/q)| + O\left((1+|\kappa|N)ND_1^2
\exp\left(-\frac{c_6\log N}{\sqrt{\log N} + \log D_1}\right)\right).
\end{equation*}
Hence, by (\ref{eqn.usefulestimateforlambda}),
\begin{equation*}
|\wh{\Lambda_{N,d}}(\theta)| \leq
\frac{|\tau_{a,d,q}|}{\phi(q)}|\wh{\Lambda_{N,d}}(0)| +
O\left((1+|\kappa|N)ND_1^2 \exp\left(-\frac{c_6\log N}{\sqrt{\log N}
+ \log D_1}\right)\right).
\end{equation*}
Now, if $(a,q)=1$ then $|\tau_{a,d,q}| \leq 1$ so we have the first
part of the proposition.

To get the lower bound on $|\wh{\Lambda_{N,d}}(0)|$ we return to
(\ref{eqn.zeroestimateforlambda}). If $x \geq 16$ and $\epsilon \in
(0,1/2]$ then $1-x^{-\epsilon}/(1-\epsilon) \geq \epsilon$, whence
\begin{equation*}
dN - \frac{(dN)^{\beta_D}}{\beta_D} \geq dN (\beta_D-1) + O(1).
\end{equation*}
Inserting the estimate for $(\beta_D-1)^{-1}$ (which we get since
$(D_1,D_0)$ is exceptional) and recalling that $d \geq d_D$ yields
the lower bound for $|\wh{\Lambda_{N,d}}(0)|$.
\end{proof}

\subsection{Unexceptional pairs}

In this subsection we assume that $(D_1,D_0)$ is unexceptional. The
argument here is easier than that for exceptional pairs and proceeds
as above except that terms involving the exceptional zero do not
occur. We omit the details.

\begin{proposition}[Major arcs estimate for unexceptional pairs]\label{prop.majorarcsestimateforunexceptionalpairs}
For every set of non-negative integers $N,a,q,d$ with $1 \leq dq
\leq D_0$, $(a,q)=1$ and $N \geq 1$, and elements $\theta \in \T$ we
have
\begin{equation*}
|\wh{\Lambda_{N,d}}(\theta) \leq
\frac{|\wh{\Lambda_{N,d}}(0)|}{\phi(q)} + O\left((1+|\kappa|N)ND_1^2
\exp\left(-\frac{c\log N}{\sqrt{\log N} + \log D_1}\right)\right)
\end{equation*}
where $\kappa:= \theta - a/q$ and
\begin{equation*}
|\wh{\Lambda_{N,d}}(0)| \geq \frac{dN}{\phi(d)} + O\left(ND_1^2
\exp\left(-\frac{c\log N}{\sqrt{\log N} + \log D_1}\right)\right).
\end{equation*}
\end{proposition}

\section{The minor arcs}\label{sec.minorarcestimate}

The minor arcs are far easier to estimate that the major arcs were.
We begin with Vinogradov's classic estimate, recalling that
$\Lambda_N$ is shorthand for $\Lambda_{N,1}$.
\begin{theorem}\emph{(\cite[Chapter 25]{HD})}\label{thm.minaorarcsestimate}
Suppose that $N$ and $q \leq Q$ are positive integers, $\theta \in
\T$ and $a \in \{1,\dots,q\}$ is coprime to $q$ and has $|\theta -
a/q| \leq 1/qQ$. Then
\begin{equation*}
|\wh{\Lambda_N(\theta)}| \ll (\log
N)^4\left(\frac{N}{\sqrt{q}}+N^{4/5} + \sqrt{Nq}\right).
\end{equation*}
\end{theorem}
This has the following relevant corollary.
\begin{corollary}[Minor arcs estimate]\label{cor.specializedminaorarcsestimate}
Suppose that $d \leq N$ and $q \leq Q$ are positive integers,
$\theta \in \T$ and $a \in \{1,\dots,q\}$ is coprime to $q$ and has
$|\theta - a/q| \leq 1/qQ$. Then
\begin{equation*}
|\wh{\Lambda_{N,d}}(\theta)| \ll d(\log
N)^4\left(\frac{N}{\sqrt{q}}+N^{4/5} + \sqrt{NQ}\right).
\end{equation*}
\end{corollary}
\begin{proof}
Begin by noting that
\begin{equation*}
\wh{\Lambda_{N,d}}(\theta) = \frac{1}{d}\sum_{m=0}^{d-1}{\sum_{x
\leq dN+1}{\Lambda(x)e\left(\theta.\frac{x-1}{d} +
\frac{m(x-1)}{d}\right)}},
\end{equation*}
so
\begin{equation*}
|\wh{\Lambda_{N,d}}(\theta)| \leq
\frac{1}{d}\sum_{m=0}^{d-1}{|\wh{\Lambda_{dN+1}}((\theta+m)/d)|}.
\end{equation*}
Now, suppose that $m \in \{0,\dots,d-1\}$ and write
$\theta':=(\theta +m)/d$. We may apply Dirichlet's pigeon-hole
principle to get a positive integer $q' \leq Q':=2dQ$ and another
$a' \in \{1,\dots,q'\}$ with $(a',q')=1$ and such that $|\theta' -
a'/q'| \leq 1/q'Q'$. So
\begin{equation*}
|a'/q' - (a+mq)/dq| \leq |\theta' - a'/q'| + |\theta' - (a+mq)/dq|
\leq 1/q'Q' + 1/dqQ,
\end{equation*}
and hence
\begin{equation*}
|a'dq - (a+mq)q'| \leq 1/2 + q'/Q.
\end{equation*}
The left hand side is an integer and if $q' < q/2$ then it is zero.
This implies that $q | q'$ since $(q,a+mq)=1$, whence $q' \geq q$.
This contradiction means that $q' \geq q/2$. Now we just apply
Theorem \ref{thm.minaorarcsestimate} to the approximation $a'/q'$
(to $\theta'$) to get the result.
\end{proof}

\section{Some energy increment lemmas}\label{sec.energyincrementlemmas}

The main result of this section is an energy increment argument.
Such arguments are common, and an example from a very similar
context may be found in \cite{ES} and \cite{DRHB}.

We begin with a preliminary technical lemma.
\begin{lemma}\label{lem.basicenergyincrementlemma}
Suppose that $P$ is an arithmetic progression with common difference
$d$ and $A \subset \{1,\dots,N\}$ has $\alpha N$ elements. Suppose,
further, that
\begin{equation*}
\sum_{x \in \Z}{(1_A - \alpha 1_{[N]}) \ast 1_{P}(x)^2} \geq
c\alpha^2N|P|^2.
\end{equation*}
Then there is an integer $x' \in \Z$ such that
\begin{equation*}
1_A \ast 1_{P}(x') \geq \alpha(1+c)|P| + O(N^{-1}d|P|^2).
\end{equation*}
\end{lemma}
\begin{proof}
First note that
\begin{eqnarray*}
\alpha\sum_{x \in \Z}{1_A \ast 1_{P}(x)1_{[N]}\ast1_{P}(x)} & = &
\alpha \sum_{x \in \Z}{1_A(x) 1_{[N]} \ast 1_{P} \ast 1_{P}(x)}\\
& = & \alpha^2N|P|^2 + O(\alpha d |P|^3)
\end{eqnarray*}
and
\begin{equation*}
\alpha^2\sum_{x \in \Z}{1_{[N]} \ast 1_{P}(x)^2}  = \alpha^2N|P|^2 +
O(\alpha^2 d|P|^3).
\end{equation*}
Expanding the hypothesis it follows that
\begin{equation*}
\sum_{x \in \Z}{1_A \ast 1_{P}(x)^2} \geq (1+c)\alpha^2N|P|^2 +
O(\alpha d|P|^3).
\end{equation*}
Now H\"{o}lder's inequality yields
\begin{equation*}
\sup_{x \in \Z}{1_A \ast 1_P(x)}\alpha N|P| \geq \sum_{x \in \Z}{1_A
\ast 1_{P}(x)^2},
\end{equation*}
from which the result follows.
\end{proof}
The next result is a standard form of the energy increment
argument. It may be useful to recall the definition of the intervals
$\mathfrak{M}_{a,q,\eta}$ from \S\ref{sec.notation} before reading
further.
\begin{proposition}\label{prop.mainenergyincrementresult}
Suppose that $\eta>0$, $N$ and $q$ are positive integers, and $A
\subset \{1,\dots,N\}$ has $\alpha N$ elements. Write
\begin{equation*}
E_{A,q,\eta}:=\alpha^{-1}|A|^{-1}\int_{\theta \in
\mathfrak{M}_{q,\eta}}{|(1_A-\alpha1_{[N]})^\wedge(\theta)|^2d\theta}.
\end{equation*}
Then there is a arithmetic progression $P$ with common difference
$q$ and $|P| \gg q^{-1}\min \{\eta^{-1},E_{A,q,\eta}|A|\}$ such that
$|A \cap P| \geq \alpha(1+E_{A,q,\eta}/4)|P|$.
\end{proposition}
\begin{proof}
Let $P$ be the progression of common difference $q$ and length
$2M+1$ centred about the origin; we shall optimize for $M$ later. In
this case (by scaling the Dirichlet kernel) we have
\begin{equation*}
\wh{1_P}(\theta)=\frac{\sin (|P|\pi q \theta)}{\sin (\pi q \theta)},
\end{equation*}
with the usual convention at $\theta=0$. Now suppose that $\theta
\in \mathfrak{M}_{q,\eta}$, so that there is some integer $a$ with
$|\theta - a/q| \leq \eta$. Thus, writing $\kappa := \theta - a/q$
and recalling the inequalities $|\sin x| \geq 2|x|/\pi$ if $|x| \leq
\pi/2$ and $|\sin x| \leq |x|$, we have
\begin{equation*}
|\wh{1_P}(\theta)| = \left|\frac{\sin (|P|\pi q \kappa)}{\sin (\pi q
\kappa)}\right|\geq \frac{|2|P|q \kappa|)}{|\pi q \kappa|} =
\frac{2|P|}{\pi}
\end{equation*}
provided $|P|q\eta \leq 1/2$. It follows that
\begin{equation*}
\int_{\theta \in
\mathfrak{M}_{q,\eta}}{|(1_A-\alpha1_{[N]})^\wedge(\theta)|^2|\wh{1_P}(\theta)|^2d\theta}
\geq \frac{4}{\pi^2}E_{A,q,\eta}\alpha |A||P|^2.
\end{equation*}
Now the left hand side is certainly dominated by the same integral
without the restriction of the domain of integration and hence, by
Parseval's theorem applied to the unrestricted domain, we have
\begin{equation*}
\sum_{x \in \Z}{(1_A - \alpha 1_{[N]}) \ast 1_P(x)^2} \geq
\frac{4}{\pi^2}E_{A,q}|A|^2|P|^2.
\end{equation*}
Now we may apply Lemma \ref{lem.basicenergyincrementlemma} to get
some $x' \in \Z$ such that
\begin{equation*}
1_A \ast 1_P(x') \geq \left(1+
\frac{4E_{A,q,\eta}}{\pi^2}\right)\alpha |P| + O(N^{-1}q|P|^2).
\end{equation*}
It follows that there is a choice of $M \gg q^{-1}\min
\{\eta^{-1},E_{A,q}|A|\}$ for which
\begin{equation*}
1_A \ast 1_P(x') \geq \left(1+ \frac{E_{A,q,\eta}}{4}\right)\alpha
|P|.
\end{equation*}
\end{proof}
For our work we shall use the following corollary which is designed
specifically for the problem we are considering.
\begin{corollary}\label{cor.specializedenergyincrementlemma}
Suppose that $N$ is a positive integer, $A \subset \{1,\dots,N\}$
has $\alpha N$ elements, $Q'\geq 1$ and $Q:=N/Q'$. For each $q$ with
$1 \leq q \leq Q'$ write
\begin{equation*}
E_{A,q}^*:=\alpha^{-1}|A|^{-1}\int_{\theta \in
\mathfrak{M}_q^*}{|(1_A-\alpha1_{[N]})^\wedge(\theta)|^2d\theta},
\end{equation*}
and suppose that
\begin{equation*}
\sum_{q=1}^{Q'}{\frac{1}{\phi(q)}E_{A,q}^*} \geq c.
\end{equation*}
Then there is an arithmetic progression $P$ with common difference
$q \leq Q'$ and $|P| \gg Q'^{-1}N\min\{Q'^{-1},\alpha c\}$ such that
$|A \cap P| \geq \alpha(1+2^{-5}c)|P|$.
\end{corollary}
\begin{proof}
For $\eta>0$ we define $E_{A,q,\eta}$ as in Proposition
\ref{prop.mainenergyincrementresult} and write
\begin{equation*}
I_{A,a,q,\eta}:=\alpha^{-1}|A|^{-1}\int_{\theta \in
\mathfrak{M}_{a,q,\eta}}{|(1_A-\alpha1_{[N]})^\wedge(\theta)|^2d\theta}.
\end{equation*}
Begin by noting that
\begin{eqnarray*}
\sum_{q=1}^{Q'}{\frac{q}{\phi(q)}E_{A,q,Q^{-1}}}  & = &
\sum_{q=1}^{Q'}{\frac{q}{\phi(q)}\sum_{r=1}^{q}{I_{A,r,q,Q^{-1}}}}\\
& = &
\sum_{q=1}^{Q'}{\frac{q}{\phi(q)}\sum_{q'h=q}{\sum_{\genfrac{}{}{0pt}{}{r'=1}{(r,q')=1}}^{q'}{I_{A,r'h,q'h,Q^{-1}}}}},
\end{eqnarray*}
but this last expression is equal to
\begin{equation*}
\sum_{q=1}^{Q'}{\frac{q}{\phi(q)}\sum_{q'h=q}{E_{A,q',Q^{-1}}^*}} =
\sum_{q'=1}^{Q'}{E_{A,q',Q^{-1}}^*\sum_{h=1}^{Q'/q'}{\frac{q'h}{\phi(q'h)}}}.
\end{equation*}
Now we also have
\begin{equation*}
\sum_{h=1}^{Q'/q'}{\frac{q'h}{\phi(q'h)}} \geq
\frac{q'}{\phi(q')}\sum_{h=1}^{Q'/q'}{1} \geq \frac{Q'}{2\phi(q')},
\end{equation*}
and so
\begin{equation*}
\sum_{q=1}^{Q'}{\frac{q}{\phi(q)}E_{A,q,Q^{-1}}}  \geq
\frac{Q'}{2}\sum_{q'=1}^{Q'}{\frac{1}{\phi(q')}E_{A,q',Q^{-1}}^*}
\geq \frac{Q'c}{2}
\end{equation*}
by hypothesis and the fact that $1/qQ \leq 1/Q$. But, it is well
known (see, for example, the book \cite{HLMRCV}) that
\begin{equation*}
\sum_{q=1}^{Q'}{\frac{q}{\phi(q)}} \leq 4Q',
\end{equation*}
so, by a trivial instance of H\"{o}lder's inequality, we conclude
that there is some $q$ with $1 \leq q \leq Q'$ such that
$E_{A,q,Q^{-1}} \geq c/8$. We now apply Proposition
\ref{prop.mainenergyincrementresult} to get the result.
\end{proof}

\section{The main iteration lemma}\label{sec.mainiterationlemma}

Our main argument is iterative -- although the eventual
proof will be by maximality -- and the central lemma
follows. Essentially it says that if none of the various input
parameters is too small and $A-A$, that is the set of differences
between elements of $A$, is disjoint from the set of all numbers of the form $(p-1)/d$,
then $A$ must have much larger density on a reasonable
sub-progression.

\begin{lemma}[Iteration lemma]\label{lem.mainiterationlemma}
Suppose that $D_1 \geq D_0 \geq 2$, $A \subset \{1,\dots,N\}$ has
$\alpha N$ elements and either
\begin{enumerate}
\item $(D_1,D_0)$ is exceptional and $d \leq D_1$ is such that $d_D |d$;
\item or $(D_1,D_0)$ is unexceptional and $d \leq D_0$.
\end{enumerate}
Then there are absolute constants $c_8,c_9,c_{10}>0$ such that at
least one of the following holds.
\begin{enumerate}
\item \emph{(Density increment)} There is an integer $d'$ such that $d'
=O(\alpha^{-2})$ and a progression $P$ of common difference $d'$ and
length at least $(c_9\alpha /d\log N)^8N$ such that $|A\cap P| \geq
\alpha(1+c_8)|P|$.
\item \label{case.iterationlemmacase2} \emph{(Structure in difference set)} $A-A$ contains a number of the form $\frac{p-1}{d}$ with $p$ a prime.
\item \label{case.iterationlemmacase3} \emph{($N$ is small)} $N \leq O(\exp(c_{10}\log^2D_1))$.
\item \label{case.iterationlemmacase4} \emph{($d$ is large or $\alpha$ is small)} \begin{enumerate}
  \item $(D_1,D_0)$ is exceptional, and $d^{-1} = O(D_1^{-c_{10}})$ or $\alpha = O(D_1^{-c_{10}})$;
  \item or $(D_1,D_0)$ is unexceptional, and $d^{-1} =
  O(D_0^{-c_{10}}/\log^2
D_1)$ or $\alpha = O(D_0^{-c_{10}}/\log^2 D_1)$.
\end{enumerate}
\item \label{case.iterationlemmacase5} \emph{($\alpha$ is small)} $\alpha = O(D_1^{-c_{10}})$.
\end{enumerate}
\end{lemma}
\begin{proof}
Throughout the proof we shall introduce constants $c,c',c'',\dots$
which will each be optimized at some later point and will end up
being absolute positive constants. The reason for this slightly
unappealing approach is that we have not been explicit about any of
the constants in the error terms we have so far produced.

Let $c>0$ be some constant to be optimized later. Either $N \leq
c\alpha^{-1}$ (and we shall, once we have shown we can choose $c$ to
be absolute, be in outcome (\ref{case.iterationlemmacase3}) or
(\ref{case.iterationlemmacase5})) or the integer $N'=\lfloor
c\alpha N\rfloor $ has $N' \geq 1$.

Irrespective of whether $(D_1,D_0)$ is exceptional or unexceptional
we have, from Propositions
\ref{prop.majorarcsestimateforexceptionalpairs} or
\ref{prop.majorarcsestimateforunexceptionalpairs}, that
\begin{equation*}
|\wh{\Lambda_{N',d}}(0)| \gg \frac{N'}{\phi(d)} + O\left(N'D_1^2
\exp\left(-\frac{c_6\log N'}{\sqrt{\log N'} + \log
D_1}\right)\right).
\end{equation*}
Now $\phi(d) \leq d \leq D_1$ so either $N \leq c^{-1}\alpha^{-1}
\exp(O(\log^2 D_1))$ (and we shall be in outcome
(\ref{case.iterationlemmacase3}) or (\ref{case.iterationlemmacase5})
again) or we have the estimate
\begin{equation*}
|\wh{\Lambda_{N',d}}(0)| \gg \frac{N'}{\phi(d)}.
\end{equation*}
Write $I$ for the interval $[N]$ and consider the inner product
\begin{equation}\label{eqn.maininnerproduct}
\langle (1_A - \alpha 1_I) \ast (1_{-A} - \alpha 1_{-I}),
\Lambda_{N',d}\rangle.
\end{equation}
If $A-A$ contains a number of the form $(p-1)/d$ for some prime $p$
then we are in outcome (\ref{case.iterationlemmacase2}) of the
lemma; consequently assume that it does not. In this case the only
integers $x$ which support a contribution in the inner product
$\langle 1_A \ast 1_{-A},\Lambda_{N',d}\rangle $ are those for which
$dx+1$ is a prime power with the power strictly bigger than one.
There are at most $O(\sqrt{dN'})$ such integers and furthermore
$\|\Lambda_{N',d}\|_{\ell^\infty(\Z)} = O(\log dN')$ and $\|1_A \ast
1_{-A}\|_{\ell^\infty(\Z)} \leq \alpha N$, whence
\begin{equation*}
\langle 1_A \ast 1_{-A},\Lambda_{N',d}\rangle = O(\alpha
N\sqrt{dN'}\log dN').
\end{equation*}
We conclude that
\begin{equation*}
 \langle 1_A \ast 1_{-A},\Lambda_{N',d}\rangle =O(\alpha |\wh{\Lambda_{N',d}}(0)|N')
\end{equation*}
unless $N \leq \exp(O(\log \alpha^{-1}D_1c^{-1}))$ (in which case we
shall be in outcome (\ref{case.iterationlemmacase3}) or
(\ref{case.iterationlemmacase5}) again).

The other terms arising from expanding out
(\ref{eqn.maininnerproduct}) are more easily handled:
\begin{equation*}
\langle 1_I \ast 1_{-A},\Lambda_{N',d} \rangle =
\wh{\Lambda_{N',d}}(0)\alpha N + O(|\wh{\Lambda_{N',d}}(0)|N'),
\end{equation*}
\begin{equation*}
\langle 1_A \ast 1_{-I},\Lambda_{N',d} \rangle =
\wh{\Lambda_{N',d}}(0)\alpha N + O(|\wh{\Lambda_{N',d}}(0)|N')
\end{equation*}
and
\begin{equation*}
\langle 1_I \ast 1_{-I},\Lambda_{N',d} \rangle =
\wh{\Lambda_{N',d}}(0)N + O(|\wh{\Lambda_{N',d}}(0)|N').
\end{equation*}
Thus it follows that
\begin{equation*}
\langle (1_A - \alpha 1_I) \ast (1_{-A} - \alpha 1_{-I}),
\Lambda_{N',d}\rangle = \alpha^2N\wh{\Lambda_{N',d}}(0)(-1 + O(c)).
\end{equation*}
Now we pick $c \gg 1$ such that
\begin{equation*}
|\langle (1_A - \alpha 1_I) \ast (1_{-A} - \alpha 1_{-I}),
\Lambda_{N',d}\rangle| \gg \alpha^2N|\wh{\Lambda_{N',d}}(0)|,
\end{equation*}
and apply Plancherel's theorem to the left hand side to get a
Fourier space expression
\begin{equation*}
\int{|(1_A -
\alpha1_{I})^\wedge(\theta)|^2|\wh{\Lambda_{N',d}}(\theta)|d\theta}
\gg \alpha^2N|\wh{\Lambda_{N',d}}(0)|.
\end{equation*}

Let $c'>0$ be another constant to be optimized later. Write
\begin{equation*}
Q':=\frac{d^4\log^8 N}{c'^2\alpha^2} \textrm{ and } Q:= N'/Q',
\end{equation*}
and
\begin{equation*}
\mathfrak{M}':=\bigcup_{q < Q'}{\mathfrak{M}_q^*} \textrm{ and }
\mathfrak{M}:=\bigcup_{Q' \leq q \leq Q}{\mathfrak{M}_q^*}.
\end{equation*}
By Dirichlet's pigeon-hole principle $\T=\mathfrak{M} \cup
\mathfrak{M}'$ and so by the triangle inequality we have
\begin{eqnarray*}
\int{|(1_A -
\alpha1_{I})^\wedge(\theta)|^2|\wh{\Lambda_{N',d}}(\theta)|d\theta}
& \leq & \int_{\theta \in \mathfrak{M}'}{|(1_A -
\alpha1_{I})^\wedge(\theta)|^2|\wh{\Lambda_{N',d}}(\theta)|d\theta}\\
& & + \int_{\theta \in \mathfrak{M}}{|(1_A -
\alpha1_{I})^\wedge(\theta)|^2|\wh{\Lambda_{N',d}}(\theta)|d\theta}.
\end{eqnarray*}
Corollary \ref{cor.specializedminaorarcsestimate} tells us that
either $N' \leq d$ (and we are in the outcome
(\ref{case.iterationlemmacase3}) or
(\ref{case.iterationlemmacase5}), since $d \leq D_1$) or else
\begin{equation*}
|\wh{\Lambda_{N',d}}(\theta)| \ll d\log^4
N'\left(\frac{N'}{\sqrt{q}}+N'^{4/5} + \sqrt{N'Q}\right),
\end{equation*}
for $\theta \in \mathfrak{M}_q^*$ and $q \leq Q$. So, if $\theta \in
\mathfrak{M}$ then
\begin{equation*}
|\wh{\Lambda_{N',d}}(\theta)| \ll \frac{N'}{d}\left(c'\alpha +
d^2N'^{-1/10} \right).
\end{equation*}
Once again, either $N' \leq  c'^{-10}d^{20}\alpha^{-10}$ (and we are
in the outcome (\ref{case.iterationlemmacase3}) or
(\ref{case.iterationlemmacase5}), since $d \leq D_1$), or we have
\begin{equation*}
|\wh{\Lambda_{N',d}(\theta)}| \ll c'\alpha|\wh{\Lambda_{N',d}}(0)|
\textrm{ for all } \theta \in \mathfrak{M}.
\end{equation*}
It follows that
\begin{eqnarray*}
 \int_{\theta \in \mathfrak{M}}{|(1_A -
\alpha1_{I})^\wedge(\theta)|^2|\wh{\Lambda_{N',d}}(\theta)|d\theta}
& \ll & c'\alpha |\wh{\Lambda_{N',d}}(0)|. \int{|(1_A -
\alpha1_{I})^\wedge(\theta)|^2d\theta}\\ & \ll & c'\alpha^2
N|\wh{\Lambda_{N',d}}(0)|,
\end{eqnarray*}
by Parseval's theorem. Hence we can choose $c' \gg 1$ so that
\begin{equation*}
\int_{\theta \in \mathfrak{M}'}{|(1_A -
\alpha1_{I})^\wedge(\theta)|^2|\wh{\Lambda_{N',d}}(\theta)|d\theta}
\gg \alpha^2 N|\wh{\Lambda_{N',d}}(0)|.
\end{equation*}
If $(D_1,D_0)$ is unexceptional and $dQ' > D_0$ then we are in
outcome (\ref{case.iterationlemmacase3}) or
(\ref{case.iterationlemmacase4}) or else we can apply Proposition
\ref{prop.majorarcsestimateforunexceptionalpairs}; if $(D_1,D_0)$ is
exceptional and $dQ' > D_1$ then we are in outcome
(\ref{case.iterationlemmacase3}) or (\ref{case.iterationlemmacase4})
or else we can apply Proposition
\ref{prop.majorarcsestimateforexceptionalpairs}. So either we are
done or we could apply the appropriate proposition and get
\begin{equation*}
|\wh{\Lambda_{N',d}}(\theta)| \leq
\frac{|\wh{\Lambda_{N',d}}(0)|}{\phi(q)} +
O\left((1+Q^{-1}N')N'D_1^2 \exp\left(-\frac{c_6\log N'}{\sqrt{\log
N'} + \log D_1}\right)\right).
\end{equation*}
In view of the definition of $Q$ this error term is
\begin{equation*}
O\left(\alpha^{-2}(\log^8 N)N'D_1^6 \exp\left(-\frac{c_6\log
N'}{\sqrt{\log N'} + \log D_1}\right)\right),
\end{equation*}
and so once again either we are in outcome
(\ref{case.iterationlemmacase3}) or (\ref{case.iterationlemmacase5})
of the lemma or
\begin{equation}\label{eqn.majorarcsestimate}
\sup_{\theta \in \mathfrak{M}_q^*}{|\wh{\Lambda_{N',d}}(\theta)|}
\ll \frac{|\wh{\Lambda_{N',d}}(0)|}{\phi(q)} \textrm{ for all } q
\leq Q'.
\end{equation}
Set $Q'':= c''^{-2}\alpha^{-2}$ for some $c'' > 0$ which will be
chosen shortly. Write
\begin{equation*}
\mathfrak{M}'':=\bigcup_{Q'' \leq q < Q'}{\mathfrak{M}_q^*} \textrm{
and } \mathfrak{M}''':=\bigcup_{q < Q''}{\mathfrak{M}_q^*},
\end{equation*}
so that $\mathfrak{M}'=\mathfrak{M}'' \cup \mathfrak{M}'''$. Then
\begin{equation*}
\int_{\theta \in \mathfrak{M}''}{|(1_A -
\alpha1_{I})^\wedge(\theta)|^2|\wh{\Lambda_{N',d}}(\theta)|d\theta}
\ll c''\alpha^2 N |\wh{\Lambda_{N',d}}(0)|
\end{equation*}
by Parseval's theorem and (\ref{eqn.majorarcsestimate}) since
$\phi(n) \gg n^{1/2}$; pick $c'' \gg 1$ so that
\begin{equation*}
\int_{\theta \in \mathfrak{M}'''}{|(1_A -
\alpha1_{I})^\wedge(\theta)|^2|\wh{\Lambda_{N',d}}(\theta)|d\theta}
\gg \alpha^2 N |\wh{\Lambda_{N',d}}(0)|.
\end{equation*}
Now, by the triangle inequality we get
\begin{equation*}
\sum_{q=1}^{Q''}{\int_{\theta \in \mathfrak{M}_q^*}{|(1_A -
\alpha1_{I})^\wedge(\theta)|^2|\wh{\Lambda_{N',d}}(\theta)|d\theta}}
\gg \alpha^2 N|\wh{\Lambda_{N',d}}(0)|,
\end{equation*}
and so by (\ref{eqn.majorarcsestimate})
\begin{equation*}
\sum_{q=1}^{Q''}{\frac{|\wh{\Lambda_{N',d}}(0)|}{\phi(q)}\int_{\theta
\in \mathfrak{M}_q^*}{|(1_A -
\alpha1_{I})^\wedge(\theta)|^2d\theta}} \gg \alpha^2
N|\wh{\Lambda_{N',d}}(0)|.
\end{equation*}
Thus by Corollary \ref{cor.specializedenergyincrementlemma} (with
the fact that $|\wh{\Lambda_{N',d}}(0)|
>0$) we find ourselves in the first case of the lemma.
\end{proof}

\section{Proof of Theorem \ref{thm.mainresult}}\label{sec.proofofmaintheorem}

The main argument is now fairly straightforward. We begin with a
preliminary technical lemma.
\begin{lemma}\label{lem.averaginglemma}
Suppose that $A \subset \{1,\dots,N\}$ has $\alpha N$ elements and
$d$ is an integer. Then there is a progression $P$ with common
difference $d$ and $|P| \gg \alpha N/d$ such that $|A\cap P| \geq
\alpha|P|/2$.
\end{lemma}
\begin{proof}
Let $P'$ be a progression with common difference $d$. Then
\begin{equation*}
\sum_{x \in \Z}{|1_{[N]} \ast 1_{P'}(x) - |P'|1_{[N]}(x)|} =
O(d|P'|^2),
\end{equation*}
whence
\begin{equation*}
|A||P'| = \langle 1_A,1_{[N]} \ast 1_{P'} \rangle + O(d|P'|^2) =
\langle 1_A \ast 1_{P'}, 1_{[N]} \rangle + O(d|P'|^2).
\end{equation*}
It follows that we can pick $|P'| \gg \alpha N/d$ such that
\begin{equation*}
\langle 1_A\ast 1_{P'} ,1_{[N]}  \rangle \geq \alpha N |P'|/2.
\end{equation*}
Thus, by H\"{o}lder's inequality, there is some translate $P$ of
$P'$ with the desired property.
\end{proof}
We now turn to the main proof.
\begin{proof}[Proof of Theorem \ref{thm.mainresult}]
Write $I$ for the interval $[N]$. We fix $D_1 \geq D_0 \geq 2$, to
be optimized at the end of the argument, and put
$D_1=D_0^{\max\{2c_{10}^{-1},1\}}$, where $c_{10}$ is the absolute
constant from Lemma \ref{lem.mainiterationlemma}. We consider two
cases according to whether $(D_1,D_0)$ is exceptional or unexceptional.\\

\noindent \textbf{($(D_1,D_0)$ is exceptional)} This gives us an
integer $d_D \leq D_0$ with a number of properties. By Lemma
\ref{lem.averaginglemma} there is a progression $P$ with common
difference $d_D$ and $|P| \gg \alpha N/d_D$ such that $|A \cap P|
\geq \alpha|P|/2$. Let $I_D:=\{1,\dots,|P|\}$ and let $A_D$ be the
affine transformation of $A \cap P$ so that it lies in $I_D$. We
write $\alpha_D$ for the density of $A$ in $I_D$ and $N_D$ for the
length of $I_D$. Thus
\begin{equation*}
\alpha_D \geq \alpha/2 \textrm{ and } N_D \gg \alpha N/D_0.
\end{equation*}
Furthermore, by the hypothesis on $A$, $A_D-A_D$ does not contain
any numbers of the form $(p-1)/d_D$ with $p$ a prime. Let $\eta>0$
be a parameter to be optimized later and let $P'$ an arithmetic
progression such that
\begin{equation*}
\alpha'N'^{\eta^2}d'^{-\eta}
\end{equation*}
is maximal, where $\alpha'$ is the relative density of $A_D$ on
$P'$, that is $|A_D \cap P'| / |P'|$,  $N'$ is the length of $P'$
and $d'$ is the common difference of $P'$. The choice of $\eta^2$
and $\eta$ is made with the benefit of hindsight; we could use two
different parameters and optimize for both at the end.

In view of the maximality of $P'$, we have $\alpha_DN_D^{\eta^2}
\leq \alpha' N'^{\eta^2}d'^{-\eta}$. Now, since $\alpha' \leq 1$,
$d' \geq 1$and $N' \leq N_D$ it follows that
\begin{equation*}
\alpha_D \leq \alpha', d' \leq \alpha_D^{-\eta^{-1}} \textrm{ and }
N' \geq \alpha_D^{\eta^{-2}}N_D,
\end{equation*}
whence
\begin{equation}\label{eqn.exceptionalvariablesbound1}
\alpha \ll \alpha', \textrm{ } \log d' \ll \eta^{-1}\log \alpha^{-1}
\end{equation}
and
\begin{equation}\label{eqn.exceptionalvariablesbound2}
\log N = \log N' + O(\eta^{-2}\log \alpha^{-1} + \log D_0).
\end{equation}
Again, let $I':=\{1,\dots,N'\}$ and $A'$ be the affine
transformation of $A_D \cap P'$ so that it lies in $I'$. Apply Lemma
\ref{lem.mainiterationlemma} to get the following possibilities.
\begin{enumerate}
\item Either there is a progression $P''$ of common difference $d'' =
O(\alpha'^{-2})$, length at least $(c_9\alpha'/d'd_D\log N')^8 N'$
such that $|A' \cap P''| \geq \alpha'(1+c_8)|P''|$;
\item or $A'-A'$ contains a number of the form $\frac{p-1}{d'd_D}$
with $p$ a prime;
\item or $N' \leq O(\exp(c_{10} \log^2 D_1))$ whence, by
(\ref{eqn.exceptionalvariablesbound2}),
\begin{equation*}
\log N \ll \eta^{-2}\log \alpha^{-1} +  \log^2 D_1;
\end{equation*} \item or $(d'd_D)^{-1} = O(D_1^{-c_{10}})$ whence, by the relationship between $D_0$ and $D_1$ we get $
d' \gg D_1^{-c_{10}/2}$ and so, by
(\ref{eqn.exceptionalvariablesbound1}),
\begin{equation*}
\log D_1 \ll \eta^{-1} \log \alpha^{-1};
\end{equation*} \item or $\alpha' = O(D_1^{-c_{10}})$.
\end{enumerate}
In the first case, the maximal way in which $P'$ was chosen ensures
that
\begin{equation*}
\alpha'N'^{\eta^2}d'^{-\eta} \geq
\alpha'(1+c_8)(N'(c_9\alpha'/d'd_D\log
N')^8)^{\eta^2}(c\alpha')^{-2\eta}d'^{-\eta},
\end{equation*}
from which we conclude
\begin{equation*}
\eta^{-2} \ll \eta^{-1}\log \alpha'^{-1} + \log D_0 + \log d' + \log
\log N'.
\end{equation*}
Inserting the bounds from (\ref{eqn.exceptionalvariablesbound1}) and
(\ref{eqn.exceptionalvariablesbound2}) and the fact that $\log D_0
\ll \log D_1$ we get
\begin{equation*}
\eta^{-2} \ll \eta^{-1}\log \alpha^{-1} + \log D_1 + \log \log N,
\end{equation*}
and hence, by solving the quadratic,
\begin{equation*}
\eta^{-1} \ll \log \alpha^{-1} + \sqrt{\log D_1} + \sqrt{\log \log
N}.
\end{equation*}
Write $C$ for the absolute constant hiding in the above expression.
We optimize $\eta$ be taking
\begin{equation*}
\eta^{-1} = 2C(\log \alpha^{-1} + \sqrt{\log D_1} + \sqrt{\log \log
N}),
\end{equation*}
and so we have derived a contradiction and must be in another of the
above cases. By assumption we are not in the second case so we
conclude that either
\begin{equation*}
\log N \ll \eta^{-2}\log \alpha^{-1} +  \log^2 D_1 \textrm{ or }
\log D_1 \ll \eta^{-1} \log \alpha^{-1}.
\end{equation*}
Inserting our choice of $\eta$ we get the either
\begin{equation}\label{eqn.exceptionaloutput}
\log N \ll \log^2D_1 \textrm{ or } \log D_1 \ll \log\alpha^{-1}(\log
\alpha^{-1} + \sqrt{\log D_1} + \sqrt{\log \log N}).
\end{equation}\\
\noindent\textbf{($(D_1,D_0)$ is unexceptional)} In this case we can
proceed directly without the aid of Lemma \ref{lem.averaginglemma}.
Let $\eta>0$ be a (new) parameter to be optimized later and let $P'$
an arithmetic progression such that
\begin{equation*}
\alpha'N'^{\eta^2}d'^{-\eta}
\end{equation*}
is maximal, where $\alpha'$ is the relative density of $A$ on $P'$,
$N'$ is the length of $P'$ and $d'$ is the common difference of
$P'$.

As before, in view of the maximality of $P'$, we have
\begin{equation}\label{eqn.unexceptionalvariablesbound1}
\alpha \ll \alpha', \textrm{ } \log d' \ll \eta^{-1}\log \alpha^{-1}
\end{equation}
and
\begin{equation}\label{eqn.unexceptionalvariablesbound2}
\log N = \log N' + O(\eta^{-2}\log \alpha^{-1}).
\end{equation}
Again, let $I':=\{1,\dots,N'\}$ and $A'$ be the affine
transformation of $A_D \cap P'$ so that it lies in $I'$. Apply Lemma
\ref{lem.mainiterationlemma} to get the following possibilities.
\begin{enumerate}
\item Either there is a progression $P''$ of common difference $d'' =
O(\alpha'^{-2})$, length at least $(c_9\alpha'/d'\log N')^8 N'$ such
that $|A' \cap P''| \geq \alpha'(1+c_8)|P''|$;
\item or $A'-A'$ contains a number of the form $\frac{p-1}{d'}$ with
$p$ a prime;
\item or $N' \leq O(\exp(c_{10} \log^2 D_1))$ whence, by
(\ref{eqn.unexceptionalvariablesbound2}),
\begin{equation*}
\log N \ll \eta^{-2}\log \alpha^{-1} +  \log^2 D_1;
\end{equation*} \item or $d'^{-1} = O(D_0^{-c_{10}}/\log^2D_1)$ whence $d' \gg D_1^{c_{10}^2/4}$ and so, by (\ref{eqn.unexceptionalvariablesbound1}),
\begin{equation*}
\log D_1 \ll \eta^{-1} \log \alpha^{-1};
\end{equation*} \item or $\alpha' = O(D_0^{-c_{10}}/\log^2D_1)$.
\end{enumerate}
The analysis proceeds much as before and we conclude that either
\begin{equation}\label{eqn.unexceptionaloutput}
\log N \ll \log^2D_1 \textrm{ or } \log D_1 \ll \log\alpha^{-1}(\log
\alpha^{-1} + \sqrt{\log \log N}).
\end{equation}\\

Write $C$ for the larger of the two constants hiding in the first
inequalities in (\ref{eqn.exceptionaloutput}) and
(\ref{eqn.unexceptionaloutput}). We optimize $D_1$ by taking $\log N
= 2C \log^2D_1$ so that we are never in the first case of either.
The result follows.
\end{proof}

\section*{Acknowledgment}
The second author would like to thank Ben Green.

\bibliographystyle{alpha}

\bibliography{master}

\begin{thebibliography}{PSS88}

\bibitem[Dav00]{HD}
H.~Davenport.
\newblock {\em Multiplicative number theory}, volume~74 of {\em Graduate Texts
  in Mathematics}.
\newblock Springer-Verlag, New York, third edition, 2000.
\newblock Revised and with a preface by Hugh L. Montgomery.

\bibitem[HB87]{DRHB}
D.~R. Heath-Brown.
\newblock Integer sets containing no arithmetic progressions.
\newblock {\em J. London Math. Soc. (2)}, 35(3):385--394, 1987.

\bibitem[Luc08]{JL}
J.~Lucier.
\newblock Difference sets and shifted primes.
\newblock {\em Acta Math. Hungar.}, 120(1-2):79--102, 2008.

\bibitem[MV07]{HLMRCV}
H.~L. Montgomery and R.~C. Vaughan.
\newblock {\em Multiplicative number theory {I}. {C}lassical theory}, volume~97
  of {\em Cambridge Studies in Advanced Mathematics}.
\newblock Cambridge University Press, Cambridge, 2007.

\bibitem[PSS88]{JPWLSES}
J.~Pintz, W.~L. Steiger, and E.~Szemer{\'e}di.
\newblock On sets of natural numbers whose difference set contains no squares.
\newblock {\em J. London Math. Soc. (2)}, 37(2):219--231, 1988.

\bibitem[Ruz82]{IZRUD}
I.~Z. Ruzsa.
\newblock Uniform distribution, positive trigonometric polynomials and
  difference sets.
\newblock In {\em Seminar on Number Theory, 1981/1982}, pages Exp. No. 18, 18.
  Univ. Bordeaux I, Talence, 1982.

\bibitem[Ruz84]{IZRP1}
I.~Z. Ruzsa.
\newblock On measures on intersectivity.
\newblock {\em Acta Math. Hungar.}, 43(3-4):335--340, 1984.

\bibitem[S{\'a}r78]{AS}
A.~S{\'a}rk{\"o}zy.
\newblock On difference sets of sequences of integers. {III}.
\newblock {\em Acta Math. Acad. Sci. Hungar.}, 31(3-4):355--386, 1978.

\bibitem[Sze90]{ES}
E.~Szemer{\'e}di.
\newblock Integer sets containing no arithmetic progressions.
\newblock {\em Acta Math. Hungar.}, 56(1-2):155--158, 1990.

\end{thebibliography}

\end{document}